\documentclass[a4paper,12pt]{article}
\usepackage{times, url}
\usepackage{xcolor}

\usepackage{tikz}
\usepackage{graphicx,amsmath,amssymb,amsthm,amsfonts,bm,float,cite}
\newtheorem{theorem}{Theorem}[section]
\newtheorem{conjecture}{Conjecture}[section]
\newtheorem{definition}{Definition}[section]

\theoremstyle{definition}
\newtheorem{example}{Example}[section]

\usepackage{graphicx}
\usepackage{longtable}

\usepackage[none]{hyphenat}

\usepackage[center]{caption}
\begin{document}
\begin{center}
{\LARGE \bf  The $p$-adic valuation of the general degree-$2$ and degree-$3$ polynomial in $2$ variables} 
\vspace{8mm}

{\Large \bf Shubham}
\vspace{3mm}

 School of Physical Sciences, Jawaharlal Nehru University\\ 
New Delhi, India \\
e-mail: \url{Shubha76_sps@jnu.ac.in}
\vspace{2mm}
\end{center}
{\bf Abstract:}  This paper investigates the $p$-adic valuation trees of degree-$2$ and degree-$3$ polynomials in two variables over any prime $p$, building upon prior research outlined in \cite{fourteen}.\\
{\bf Keywords:} $p$-adic valuation, valuation tree, polynomial sequences. \\ 
{\bf 2020 Mathematics Subject Classification:}11B99, 11D88, 11B83,11G16
\vspace{5mm}

\section{Introduction}
This paper explores the $p$-adic valuation tree of degree-$2$ and degree-$3$  polynomials in two variables, building upon the foundational work discussed in \cite{fourteen}. We extend the results of \cite{fourteen} to encompass degree-$2$ and degree-$3$ polynomials in two variables over any prime $p$, leveraging the $p$-adic valuation tree to analyze the $p$-adic valuations of sequences. We say the $p$-adic valuation $\nu_p(f(x,y))$ of a polynomial $f(x,y)$ has a \textit{closed-form} if its $p$-adic valuation tree has no infinite branch.
In \cite{fourteen}, it was established that the $p$-adic valuation $\nu_p(f(x, y))$ admits a closed-form when the equation $f(x, y) = 0$ has no solution in $\mathbb{Q}_p \times \mathbb{Q}_p$. Additionally, it was shown that the set $V_p(f) = {\nu_p(f(n,m)) : 0 \leq m \leq n, n \in \mathbb{N}}$, where $f(x, y) \in \mathbb{Z}[x, y]$, exhibits periodic behavior if and only if $f$ has no zeros in $\mathbb{Z}_p \times \mathbb{Z}_p$. In the periodic case, the minimal period is a power of $p$.
The main contributions of this paper, encapsulated in Theorems \ref{theorem 3.1} and \ref{theorem 4.1}, focus on the $p$-adic valuation trees of degree-$2$ and degree-$3$ polynomials in two variables, respectively. Theorem \ref{theorem 3.1} details the structure of the $p$-adic valuation tree for degree-$2$ polynomials, while Theorem \ref{theorem 4.1} extends this analysis to degree-$3$ polynomials.
The subsequent section provides an overview of the applications and known results concerning the $p$-adic valuations of various sequences.

\section{Related Work}
For $n \in \mathbb{N}$, the exponent of the highest power of a prime $p$ that divides $n$ is called the $p$-adic valuation of the $n$.  This is denoted by $\nu_p(n)$.  Legendre establishes the following result about $p$-adic valuation of $n!$ in \cite{one} \begin{center}{
     $\nu_p(n!) = \sum\limits_{k=1}^{\infty}\left \lfloor \frac{n}{p^k} \right\rfloor = \frac{n-s_p(n)}{p-1}$,}
\end{center}where $s_p(n)$ is the sum of digits of $n$ in base $p$.  It is observed in \cite{two} that $2$-adic valuation of central binomial coefficient is   $s_2(n)${, that is \begin{center}
$\nu_2(a_n)$ = $s_2(n)$, where $a_n$ = $\binom{2n}{n}$.
\end{center}}It follows from here that $a_n$ is always an even number and $a_n/2$ is odd when $n$ is a power of $2$.  This expression is called a \textit{closed form} in \cite{two}. The definition of closed form depends on the context.  This has been discussed in \cite{three, four}. 
\par The work presented in \cite{two} forms part of a general project initiated by Moll and his co-authors to analyse the set \begin{center}
 $V_x = \{ \nu_p(x_n):n \in \mathbb{N}\}$,
\end{center}for a given sequence $x =\{x_n\} $ of integers.  \\ The $2$-adic valuation of $\{n^2-a\}$ is studied in \cite{two}. It is shown that $n^2-a, a\in \mathbb{Z}$ has a simple closed form when $a \not \equiv 4,7 \mod 8$.  For these two remaining cases the valuation is quite complicated.  It is studied by the notion of the \textit{valuation tree}. 

Given a polynomial $f(x)$ with integer coefficients, the sequence $\{\nu_2(f(n))\}$ is described by a tree. This is called the valuation tree attached to the polynomial $f$.  The vertices of this tree correspond to some selected subsets \begin{center}
$C_{m,j} = \{2^mi+j: i \in \mathbb{N}\},\quad 0\leq j<2^m$
\end{center}starting with the root node $C_{0,0} = \mathbb{N}$.  The procedure to select classes is explained below in Example \ref{example: zero}.  Some notation for the vertices of this tree are introduced as follows: \begin{definition} A residue class $C_{m,j}$ is called \textit{terminal} if $\nu_2(f(2^mi+j))$ is independent of $i$.  Otherwise, it is called \textit{non-terminal}.  The same terminology is given to vertices corresponding to the class $C_{m,j}$.  In the tree,  terminal vertices are labelled by their constant valuation and non-terminal vertices are labelled by a $*$. \end{definition}
{\begin{example}\label{example: zero} Construction of valuation tree of $x^2 + 5$ is as follows:
note that $\nu_2(1+5) = 1$ and $\nu_2(2+5$) is $0$. So node $v_0$ is non terminating.  Hence it splits into two vertices and forms the first level.  These vertices correspond to the residue classes $C_{1,0}$ and $C_{1,1}$.  We can check that both these nodes are terminating with valuation $0$ and $1$. The valuation tree of $x^2+5$ is shown in Figure \ref{fig:1}.\end{example}
 \begin{figure}   \begin{center}  \begin{tikzpicture}
      [sibling distance=8em,level distance=3em,
      every node/.style={shape=circle,draw= black,align=right}]
        \node{$v_0$}
        child{node{0}
        } 
        child{node{1}}
              ;
        
      \end{tikzpicture}\end{center}
       \caption{The valuation tree of $x^2+5$} \label{fig:1}   \end{figure}
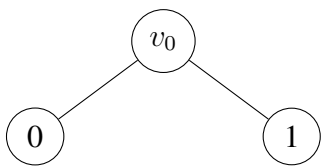 \par The main theorem of \cite{two} is as follows:\begin{theorem}
      Let $v$ be a non-terminating node at the $k$-th level for the valuation tree of $n^2 + 7$. Then $v$ splits into two vertices at the $(k + 1)$-level.  Exactly one of them terminates, with valuation $k$. The second one has valuation at least $k + 1$.
      \end{theorem}
     \par The $2$-adic valuation of the Stirling numbers is discussed in \cite{five}. The numbers $S(n,k)$ are the number of ways to partition a set of $n$ elements into exactly $k$ non-empty subsets where $n \in \mathbb{N}$ and $0\leq k \leq n$. These are explicitly given by \begin{center}
      $S(n,k) = \frac{1}{k!}\sum\limits_{i=0}^{i=k-1}(-1)^{i}\binom{k}{i}(k-i)^{n}$,
\end{center}or, by the recurrence\begin{center}
$S(n,k) = S(n-1,k-1) + kS(n-1,k)$,
\end{center}with initial condition $S(0,0) = 1$ and $S(n,0) = 0 $ for $n>0$.  \\The $2$-adic valuation of $S(n,k)$ can be easily determined for $1\leq k \leq 4$ and closed form expression is given as follows:\begin{center}
$v_2(S(n,1)) =0 = v_2(S(n,2))$,\\
\end{center}
      \begin{center}
        \[
   v_2(S(n,3))= 
\begin{cases}
   1,& \text{if } n \hspace*{.2cm}  \text{is}\hspace*{.2cm} \text{even}\\
    0,              & \text{otherwise}
\end{cases}
\]
      \end{center}
     \begin{center}
        \[
   v_2(S(n,4))= 
\begin{cases}
   1,& \text{if } n \hspace*{.2cm} \text{is}\hspace*{.2cm} \text{odd}.\\
    0,              & \text{otherwise}
\end{cases}
\]
      \end{center}
      The important conjecture described there is that the partitions of $\mathbb{N}$ in classes of the form \begin{center} $C_{m,j}^{(k)} = \{2^mi+j:$ $i \in \mathbb{N}$ and starts at the point where $2^mi+j \geq k$ \} \end{center} leads to a clear pattern for $v_2(S(n,k))$ for $k$ $\in \mathbb{N}$ is fixed. We recall that the parameter $m$ is called the level of the class.   The main conjecture of \cite{five} is now stated:
      \begin{conjecture}Let $k \in\mathbb{N}$ be fixed.  Then we conjecture that \begin{itemize}
 \item[(a)] 
there exists a level $m_0(k)$ and an integer $\mu(k)$ such that for any $m\geq m_0(k)$,  the number of non-terminal classes of level $m$ is $\mu(k)$, independently of $m$;
\item[(b)] moreover, for each $m \geq m_0(k)$, each of the $\mu(k)$ non-terminal classes splits into one terminal and one non-terminal subclass. The latter generates the next level set.
   \end{itemize}
      \end{conjecture}This conjecture is only established for the case $k = 5$.
      A similar conjecture is given in \cite{six} for the $p$-adic valuation of the Stirling numbers.\par From the point of view of application, the $2$-adic valuation tree is used to solve quadratic and cubic diophantine equations in \cite{eight}. Apart from this,  some other special cases are studied in \cite{nine,ten,eleven}.
     \par In \cite{fourteen},  the set \begin{center}
      $V_f$ = $\{ \nu_p(f(m,n)): m,n \in \mathbb{N}\}$,
\end{center}where $f(x,y) \in \mathbb{Z}[x,y]$,  is discussed by the generalized notion of the \textit{valuation tree}.  The definition of the $p$-adic valuation tree in \cite{fourteen} is as follows:         
 \begin{definition} Let $p$ be a prime number. Consider the integers $f(x,y)$ for every $(x,y)$ in $\mathbb{Z}^2$. The $p$-adic\textit{ valuation tree} of $f$ is a rooted, labelled $p^2$-ary tree defined recursively as follows:
\begin{enumerate}

   \item[1.] Suppose that  $v_{0}$ be a root vertex at level $0$. There are $p^2$ edges from this root vertex to its $p^2$ children vertices at level $1$. These vertices correspond to all possible residue classes  $(i_{0},j_{0}) \mod p$.  Label the vertex corresponding to the class $(i_{0},j_{0}) $ with $0$ if $f(i_{0},j_{0}) \not\equiv 0 \pmod p$ and with $*$ if $f(i_{0},j_{0}) \equiv 0 \pmod p$.  \item[2.] If the label of a vertex is $0$, it does not have any children. 
    
    \item[3.] If the label of a vertex is $*$, then it has $p^2$ children at level $2$. These vertices correspond to the residue classes $ (i_{0} + i_{1}p, j_{0} + j_{1}p)\mod p^2$ where $i_{1},j_{1} \in\{0,1,2,\dots ,p-1\} $ and $(i_{0},j_{0}) \mod p$ is the class of the parent vertex.\end{enumerate}    
     
     This process continues recursively so that at the $l^{\textrm{th}}$ level, there are $p^2$ children of any non-terminating vertex in the previous level $(l-1)$, each child of which corresponds to the residue classes  $(i_{0} + i_1p +\cdots+i_{l-1}p^{l-1}, j_{0}+j_1p+ \cdots + j_{l-1}p^{l-1})\pmod{p^l}$.  Here $i_{l-1}, j_{l-1}$ $\in \{0,1,2,\dots,p-1\}$ and $(i_{0} + i_1p +\cdots + i_{l-2}p^{l-2}, j_{0} + j_1p + \cdots + j_{l-2}p^{l-2}) \pmod{p^{l-1}}$ is the class of the parent vertex.  Label the vertex corresponding to the class $(i,j)$ with $l-1$ if $f(i,j) \not \equiv 0 \pmod {p^l}$ and $*$ if $f(i,j)  \equiv 0 \pmod {p^l}$.  Thus $ i = i_{0} + i_1p + \dots +i_{l-1}p^{l-1},  j = j_{0} + j_1p + \dots +j_{l-1}p^{l-1}$. \end{definition}
    
     \begin{example} Valuation tree of $x^2 + y^2 + xy +x + y +1$ is shown in Figure \ref{fig:2}.\end{example}
\begin{figure}     \begin{center} \begin{tikzpicture}
      [sibling distance=4em,level distance=2em,
      every node/.style={shape=circle,draw= black,align=right}]
        \node{}
        child{node{0}} 
        child{node{0}}
        child{node{0}}
       child{node{*}child{node{1}}
               child{node{1}}
               child{node{1}}
               child{node{1}}}        ;
        
      \end{tikzpicture}\end{center}   \caption{The valuation tree of $x^2 + y^2 + xy +x + y +1$}\label{fig:2}\end{figure}
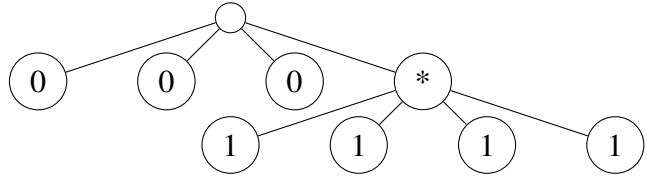

      So \begin{center}
        \[
   \nu_2(x^2+y^2+xy+x+y+1)= 
\begin{cases}
   1,& \text{if both $x, y$ are odd.} \\
    0,              & \text{otherwise}
\end{cases}
\]
      \end{center}
\section{$p$-adic Valuation Tree of Degree-$2$ Polynomial}
The $2$-adic valuation tree is given in \cite{fourteen}.  We extend this result for general prime $p$.
\begin{theorem}\label{theorem 3.1}
Consider the $p$-adic valuation tree of the general bivariate polynomial$f(x, y) \in \mathbb{Z}[x,y]$ of degree two, and let $v$ be a vertex labelled with $*$ at level $k \geq 1$, $p > 2$. Then, at the next level $(k+1)$, vertex $v$ splits into $p^2$ vertices. There are following possibilities about the labelling of these nodes:
\begin{itemize}
\item[1.] All nodes are labelled with $*$.
\item[2.] All nodes are labelled with $k+1$.
\item[3.] Exactly $p$ nodes are labelled with $*$ and rest $p^2-p$ are labelled with $k+1$.
\end{itemize}
\end{theorem}
\begin{proof}
Let $f(x,y) = \sum_{l+m \leq 2} a_{lm}x^ly^m, l, m \in \mathbb{Z}$ be the general degree-$2$ polynomial in two variables and $(b_{k-1},c_{k-1})$ be associated with the vertex $v$ at the $k$-th level of the valuation tree.  So we have $$
   f(b_{k-1},c_{k-1}) \equiv 0 \pmod {p^k}, $$ where
    $b_{k}= (i_{k}i_{k-1}\cdots i_1i_0)_p  = b_{k-1} + p^ki_k,
      c_{k}= (j_{k}j_{k-1}\cdots j_1j_0)_p = b_{k-1} + p^ki_k$. Here $i_{0},i_{1},\dots, i_{k}, j_{0},j_{1},\dots,j_{k} \in \{0,1,\dots,p-1\}$ and $(b_0,c_0) = (i_0,j_0)$. We want to find $(i_k,j_k)$ such that \begin{equation}
   f(b_k,c_k) \equiv 0 \pmod {p^{k+1}}. 
    \end{equation} On putting the expression for $(b_k,c_k)$ into above equation we get
    \begin{equation}
    \sum_{l+m \leq 2} a_{lm}(b_{k-1} + p^ki_k)^l(c_{k-1} + p^kj_k)^m \equiv 0 \pmod p^{k+1}
    \end{equation} On solving above equation and using $f(b_{k-1}, c_{k-1}) = \alpha p^k$ for some integer $\alpha$, we get \begin{equation}
    \alpha + 2a_{20}i_kb_{k-1} + a_{11}(i_kc_{k-1} + j_kb_{k-1}) + a_{10}i_k + a_{01}j_k + 2a_{02}j_kc_{k-1} \equiv 0 \pmod p.
    \end{equation}Further by using $b_{k-1} \equiv i_0 \pmod p$ and $c_{k-1} \equiv j_0 \pmod p$ in the above equation, we get \begin{equation}\label{equ 4}
    \alpha + (2a_{20}i_0 + a_{11}j_0 + a_{10})i_k + (2a_{02}j_0 + a_{11}i_0 + a_{01})j_k \equiv 0 \pmod p.
    \end{equation}Based on the values of $(i_0,j_0),$ we can have following cases:
    \begin{itemize}
    \item[Case 1.] If $(i_0,j_0) = (0,0)$ then equation \eqref{equ 4} becomes \begin{equation}
     \alpha + a_{10}i_k + a_{01}j_k \equiv 0 \pmod p.
    \end{equation} In this case if atleast one of $a_{10}$ or $a_{01}$ has $p$-adic valuation zero, then the above equation has $p$ solutions. Hence there are $p$ nodes with label $*$ and $p^2-p$ with label $k+1$ at the level $k+1$. Further if $v_p(a_{10})$ and $v_p(a_{01})$ are greater than one then depending upon $\alpha$, either all nodes are labelled with $*$ or $k+1$ at the level $k+1$.
    \item[Case 2] If $(i_0,j_0) \neq (0,0)$ and $\nu_p(a_{ij}) \geq 1$ then equation \eqref{equ 4} becomes \begin{equation}
    \alpha \equiv 0 \pmod p. 
    \end{equation}
    Hence in this case based upon $\alpha$, either all nodes at level $k+1$ are labelled with $*$ or $k+1$. 
    \item[Case 3.] if  $(i_0,j_0) \neq (0,0)$ and exactly one of $a_{20}, a_{11}, a_{10},a_{02}, a_{11}, a_{01}$ has $p$-adic valuation zero then equation \eqref{equ 4} has exactly $p$ solutions. Hence in this case also exactly $p$ nodes are labelled with $*$ and rest $p^2-p$ are labelled with $k+1$.
    \end{itemize}
    The above three cases constitutes the proof of the theorem.
\end{proof}
\section{$p$-adic valuation tree of degree-$3$ polynomial}
The following theorem gives $p$-adic valuation tree of the general degree-$3$ polynomial for all primes $p$.
\begin{theorem}\label{theorem 4.1}
Consider the $p$-adic valuation tree of the general bivariate polynomial$f(x, y) \in \mathbb{Z}[x,y]$ of degree three, and let $v$ be a vertex labelled with $*$ at level $k \geq 1$. Then, at the next level $(k+1)$, vertex $v$ splits into $p^2$ vertices. There are following possibilities about the labelling of these nodes:
\begin{itemize}
\item[1.] All nodes are labelled with $*$.
\item[2.] All nodes are labelled with $k+1$.
\item[3.] Exactly $p$ nodes are labelled with $*$ and rest $p^2-p$ are labelled with $k+1$.
\end{itemize}
\end{theorem}
\begin{proof}
Let $f(x,y) = \sum_{l+m \leq 3} a_{lm}x^ly^m, l, m \in \mathbb{Z}$ be the general degree-$2$ polynomial in two variables and $(b_{k-1},c_{k-1})$ be associated with the vertex $v$ at the $k$-th level of the valuation tree.  So we have $$
   f(b_{k-1},c_{k-1}) \equiv 0 \pmod {p^k}, $$ where
    $b_{k}= (i_{k}i_{k-1}\cdots i_1i_0)_p  = b_{k-1} + p^ki_k,
      c_{k}= (j_{k}j_{k-1}\cdots j_1j_0)_p = b_{k-1} + p^ki_k$. Here $i_{0},i_{1},\dots, i_{k}, j_{0},j_{1},\dots,j_{k} \in \{0,1,\dots,p-1\}$ and $(b_0,c_0) = (i_0,j_0)$. We want to find $(i_k,j_k)$ such that \begin{equation}
   f(b_k,c_k) \equiv 0 \pmod {p^{k+1}}. 
    \end{equation} On putting the expression for $(b_k,c_k)$ into above equation we get
    \begin{equation}
    \sum_{l+m \leq 2} a_{lm}(b_{k-1} + p^ki_k)^l(c_{k-1} + p^kj_k)^m \equiv 0 \pmod p^{k+1}
    \end{equation} On solving above equation and using $f(b_{k-1}, c_{k-1}) = \alpha p^k$ for some integer $\alpha$, we get \begin{equation}\label{equ 11}
    \begin{split}
    \alpha + (3a_{30}i_0^2 + 2a_{12}i_0j_0 + a_{12}j_0^2 + 2a_{20}i_0 + a_{11}j_0 + a_{10})i_k +  \\
     (3a_{03}j_0^2 + 2a_{21}i_0j_0 + a_{21}i_0^2 + 2a_{02}j_0 + a_{11}i_0 + a_{01})j_k \equiv 0 \pmod p
    \end{split}
    \end{equation}
First we give proof for the prime $p=2$.  In this case, equation \eqref{equ 11} becomes \begin{equation}\label{equ 12}
\alpha + (a_{30}i_0 + a_{12}j_0  + a_{11}j_0 + a_{10})i_k + 
     (a_{03}j_0 + a_{21}i_0  + a_{11}i_0 + a_{01})j_k \equiv 0 \pmod p.
\end{equation}Again we can split this case into following sub-cases:
\begin{itemize}
\item[Case 1.] If $(i_0,j_0) = (0,0)$ then the above equation becomes \begin{equation}
\alpha + a_{10}i_k + a_{01}j_k \equiv 0 \pmod p.
\end{equation}In this case if atleast one of $a_{10}$ or $a_{01}$ has $p$-adic valuation zero, then the above equation has $p$ solutions. Hence there are $p$ nodes with label $*$ and $p^2-p$ with label $k+1$ at the level $k+1$. Further if $v_p(a_{10})$ and $v_p(a_{01})$ are greater than one then depending upon $\alpha$, either all nodes are labelled with $*$ or $k+1$ at the level $k+1$.
\item[Case 2.] If $(i_0,j_0) \neq (0,0)$ and atleast one of the coefficient is not divisible by $p$ then exactly $p$ nodes are labelled by $*$ and rest are labelled by $k+1$. 
\item[Case 3.] If $(i_0,j_0) \neq (0,0)$ and all coefficients are divisible by $p$ then equation \eqref{equ 12} becomes \begin{equation}
\alpha \equiv 0 \pmod p.
\end{equation}Hence in this case either all nodes are labelled with $*$ or all are labelled with $k+1$. 
\end{itemize}
These three constitute the proof for the prime $p=2$. Now we give the proof for the $p=3$. In this case equation \eqref{equ 11} becomes \begin{equation}\label{equ 15}
\alpha + (2a_{12}i_0j_0 + a_{12}j_0^2 + 2a_{20}i_0 + a_{11}j_0 + a_{10})i_k + 
     (2a_{21}i_0j_0 + a_{21}i_0^2 + 2a_{02}j_0 + a_{11}i_0 + a_{01})j_k \equiv 0 \pmod p 
\end{equation}We split this case into the following sub-cases:\begin{itemize}
\item[Case 1.] If $(i_0,j_0) = (0,0)$, then above equation becomes\begin{equation}
\alpha + a_{10}i_k + a_{01}j_k \equiv 0 \pmod p.
\end{equation}Now if $\nu_p(a_{10}) > 0$ and $ \nu_p(a_{01})>0$, then all nodes at level $k+1$ are labelled with $k+1$ or all nodes are labelled with $*$.  If either $\nu_p(a_{10}) = 0$ or $ \nu_p(a_{01})= 0$ or both are zero then, then there are exactly $p$ nodes with label $*$ and the rest $p^2-p$ are labelled with $k+1$.
\item[Case 2.] If $ i_0 = 0$ and $ j_0 \neq 0$ then equation \eqref{equ 15} has $p$ solutions for $(i_k,j_k)$ if atleast one of coefficients $a_{ij}$ has valuation greater than zero, thus there are $p$ nodes with label $*$ and $p^2-p$ nodes with label $k+1$. If $\nu_p(a_{ij}) > 0$ then all nodes at level $k+1$ are either labelled with $*$ or all are labelled with $k+1$.
\item[Case 3] Similar reasoning for the case $ j_0 = 0$ and $ i_0 \neq 0$ gives the labels of children of the vertex $v$. 
\end{itemize}
Now we give the proof for the general prime $p > 3$.  We split this case into three cases: \begin{itemize}
\item[Case 1.] If $(i_0,j_0) = (0,0)$, then equation \eqref{equ 11} becomes\begin{equation}
\alpha + a_{10}i_k + a_{01}j_k \equiv 0 \pmod p.
\end{equation}Now if $\nu_p(a_{10}) > 0$ and $ \nu_p(a_{01})>0$, then all nodes at level $k+1$ are labelled with $k+1$ or all nodes are labelled with $*$.  If either $\nu_p(a_{10}) = 0$ or $ \nu_p(a_{01})= 0$ or both are zero then, then there are exactly $p$ nodes with label $*$ and the rest $p^2-p$ are labelled with $k+1$.
\item[Case 2.] If $ i_0 = 0$ and $ j_0 \neq 0$ then equation \eqref{equ 11} has $p$ solutions for $(i_k,j_k)$ if atleast one of coefficients $a_{ij}$ has valuation greater than zero, thus there are $p$ nodes with label $*$ and $p^2-p$ nodes with label $k+1$. If $\nu_p(a_{ij}) > 0$ then all nodes at level $k+1$ are either labelled with $*$ or all are labelled with $k+1$.
\item[Case 3] Similar reasoning for the case $ j_0 = 0$ and $ i_0 \neq 0$ gives the labels of children of the vertex $v$. 
\end{itemize}    
\end{proof}

	\makeatletter
	\renewcommand{\@biblabel}[1]{[#1]\hfill}
	\makeatother


\begin{thebibliography}{99}


 \bibitem{ten}
Amdeberhan, T., De Angelis, V., \& Moll, V. H. (2013). \textit{Complementary Bell numbers: Arithmetical properties and Wilf's conjecture.  Advances in Combinatorics},  23--56,  Springer. 

\bibitem{five}
 Amdeberhan, T., Manna, D., \& Moll, V. H.  (2008). The $2$-adic valuation of Stirling numbers.  \textit{Experimental Mathematics}, {17},  69--82. 


 \bibitem{six}
Berribeztia,  A.,   Medina, L.,  Moll, A., Moll, V., \& Noble, L.  (2010). The $p$-adic valuation of
Stirling numbers.  \textit{Journal for Algebra and Number Theory Academia}, {1}, 1--30.

\bibitem{three}
 Borwein, J. M., \& Crandall, R. (2013). Closed forms: what they are and why we care.  \textit{Notices American Mathematical Society}, 60, 50--65. 

\bibitem{eight}
Brucal--Hallare,  M., Goedhart,  E. G., Riley,  R. M., Sharma, V., \& Thompson, B. (2021). \textit{Solving quadratic and cubic Diophantine equations using $2$-adic valuation trees}.  Available online at: \url{https://doi.org/10.48550/arXiv.2105.03352}.

\bibitem{two}
 Byrnes, A., Fink, J.,   Lavigne,   G.,   Nogues,  I.,   Rajasekaran,  S.,    Yuan,  A.,   Almodovar,  L.,   Guan,  X.,   Kesarwani,  A.,   Medina,  L., Rowland, E., \& Moll, V.  H. (2019).  A closed-form solution might be given by a tree.  Valuations of quadratic polynomials.  \textit{Scientia,  Series A: Mathematical Sciences}, {29} 11--28.

\bibitem{nine}
Caicedo, J. B., Moll,  V. H., Ram\'irez, J. L., \& Villamizar, D. (2019). Extensions of set partitions and permutations. \textit{Electronic Journal of Combinatorics}, {26(2)}, Article Number P2.20.

\bibitem{four}
Chow,  T. Y.  (1999). What is a closed-form number? 
\textit{American Mathematical Monthly}, 106(5),  440--448. 

\bibitem{seven}
Conrad,  K.   \textit{A multivariable Hensel's lemma}.  Available online at: \url{https://kconrad.math.uconn.edu/blurbs/gradnumthy/multivarhensel.pdf}.

\bibitem{twelve}
Gouvêa, F. Q. (2020). \textit{$p$-adic Numbers: An Introduction} (Universitext). (3rd ed.). Springer.

 
\bibitem{one}
Legendre, A.  M. (1830). \textit{Th\'eorie des Nombres}.  Firmin Didot Fr\`eres, Paris.

\bibitem{eleven}
Sun, X., \& Moll, V. H.  (2010). A binary tree representation for the $2$-adic valuation of a sequence arising from a rational integral. \textit{Integers}, 10(2), 211--222.
\bibitem{thirteen}
Maila Brucal-Hallare, Eva G. Goedhart, Ryan Max Riley, Vaishavi Sharma, and Bianca Thompson. Solving quadratic and cubic diophantine equation using $2$-adic valuation tree.  Available online at: \url{
https://doi.org/10.48550/arXiv.2105.03352}.
\bibitem{fourteen} Shubham (2023). The $2$-adic valuation of the general degree-$2$ polynomial in $2$ variables. \textit{Notes on Number Theory and Discrete Mathematics.}29(4), 737-751, DOI: 10.7546/nntdm.2023.29.4.737-751.
\end{thebibliography}
\end{document}